\documentclass[12pt,oneside,reqno]{amsart}
\usepackage[all]{xy}
\usepackage{amsfonts,amsmath,oldgerm,amssymb,amscd,comment}
\UseComputerModernTips
\numberwithin{equation}{section}
\usepackage[breaklinks]{hyperref}
\allowdisplaybreaks

\usepackage{enumerate}

\usepackage{caption} 
\newtheorem{theorem}{Theorem}[section]

\newtheorem{lemma}[subsection]{{\bf Lemma}}
\newtheorem{coro}[subsection]{{\bf Corollary}}
\newtheorem{definition}[subsection]{Definition}

\begin{document}

	\title[ Infinite families of congruences for $2$ and $13$-core partitions ]{ Infinite families of congruences for $2$ and $13$-core partitions } 
	
	\author[ Ankita Jindal]{ Ankita Jindal}
	\address{Ankita Jindal, Indian Statistical Institute, 8th Mile, Mysore Road, RVCE Post, Bangalore, 560059.}
	\email{ankitajindal1203@gmail.com }
	
	\author[N.K. Meher]{N.K. Meher}
	\address{Nabin Kumar Meher, Department of Mathematics, Indian Institute of Information Technology, Raichur, Government of Engineering College, Yermarus Campus,Raichur, karnataka 584135, India.}
	\email{mehernabin@gmail.com, nabinmeher@iiitr.ac.in}
	
	\maketitle
	\pagenumbering{arabic}
	\pagestyle{headings}
	\begin{abstract}
		A partition of $n$ is called a $t$-core partition if none of its hook number is divisible by $t.$ In 2019, Hirschhorn and Sellers \cite{Hirs2019} obtained a parity result for $3$-core partition function $a_3(n)$. Motivated by this result, both the authors \cite{MeherJindal2022} recently proved that for a non-negative integer $\alpha,$ $a_{3^{\alpha} m}(n)$ is almost always divisible by arbitrary power of $2$ and $3$ and $a_{t}(n)$ is almost always divisible by arbitrary power of $p_i^j,$ where $j$  is a fixed positive integer and $t= p_1^{a_1}p_2^{a_2}\ldots p_m^{a_m}$ with primes $p_i \geq 5.$ In this article, by using Hecke eigenform theory, we obtain infinite families of congruences and multiplicative identities for $a_2(n)$ and $a_{13}(n)$ modulo $2$ which generalizes some results of Das \cite{Das2016}.
	\end{abstract}
	\maketitle
	
	\section{Introduction}
	A partition $\beta=(\beta_1,\beta_2, \cdots, \beta_{r})$ of $n$ is a non-increasing sequence of positive integers whose sum is $n$ and the positive integers $\beta_i$ are called parts of the partition $\beta$. A partition $\beta$ of $n$ can be represented by the Young diagram $[\beta]$ (also known as the Ferrers graph) which consists of the $s$ number of rows such that the $i^{th}$ row has $\beta_i$ number of dots $\bullet$ and all the rows start in the first column. An illustration of the Young diagram for $\beta=(\beta_1,\beta_2, \cdots, \beta_{r})$ is as follows.
	\begin{center}
		$[\beta]$:=\begin{tabular}{lcll}
			$\bullet$ & $\bullet$& $\cdots$$\cdots$$\cdots$\hspace{0.3cm}$\bullet$ & $\beta_1$ dots \\
			$\bullet$ & $\bullet$& $\cdots$$\cdots$\hspace{0.3cm}$\bullet$ & $\beta_2$ dots \\
			& $\vdots$& & \vdots \\
			$\bullet$ &$\bullet$ & $\cdots$\hspace{0.3cm}$\bullet$ & $\beta_r$ dots
		\end{tabular}
	\end{center}
	For $1\leq i \leqslant r$ and $1\leq j \leq \beta_i$, the dot of $[\beta]$ which lies in the $i^{th}$ row and $j^{th}$ column is denoted by $(i,j)^{th}$-dot of $\beta$. Let $\beta_j^{'}$ denote the number of dots in $j^{th}$ column. The hook number $H_{i,j}$ of $(i,j)^{th}$-dot is defined by $\beta_i+\beta_j^{'}-i-j+1$. In other words, $H_{i,j}=1+h_0$ where $h_0$ is the sum of the number of dots lying right to the $(i,j)^{th}$-dot in the $i^{th}$ row, the number of dots lying below the $(i,j)^{th}$-dot in the $j^{th}$ column. Given a partition $\beta$ of $n$, we say that it is a $t$-core partition if none of its hook number is divisible by $t$.

	\noindent\textbf{Example 1.} The Young diagram of the partition $\beta=(6,3,1)$ of $10$ is
	
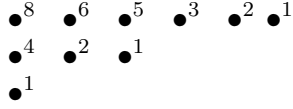
\begin{center}
		\begin{tabular}{ccccc}
			$\bullet^{8}$ & $\bullet^{6}$& $\bullet^{5}$ & $\bullet^{3}$ & $\bullet^{2}$ $\bullet^{1}$\\
			$\bullet^{4}$ & $\bullet^{2}$ & $\bullet^{1}$ & &\\
			$\bullet^{1}$ & & & &
		\end{tabular}
	\end{center}
	where the superscript on each dot represents its hook number. It can be easily observed that this is a $t$-core partition of $10$ for $t\in\{7\}$ and $t \geqslant 9$.
	
	\noindent\textbf{Example 2.} There are no $3$-core partitions of $7$. This can be easily verified by looking at the Young diagram of each partition of $7$.
	
	For a positive integer $n$, let $a_t(n)$ denote the number of $t$-core partitions of $n.$ Its generating function is given by
	\begin{equation}\label{eq1}
		\sum_{n=0}^{\infty} a_t(n) q^n = \prod \limits_{n=1}^{\infty}\frac{(1-q^{tn})^t}{(1-q^n)} = \frac{(q^t;q^t)^t_{\infty}}{(q;q)_{\infty}},
	\end{equation}
	where $(a;q)_{\infty}=(1-a)(1-aq)(1-aq^2)\cdots$.
	
	In \cite[Corollary 1]{Garvan1990}, Garvan, Kim, Stanton obtained the congruence
	\begin{align}\label{eq5711}
		a_p(p^jn-\delta_{p})\equiv 0\pmod {p^j}
	\end{align}
	where $p \in \{5,7,11\}$, $n$, $j$ are positive integers and $\delta_{p}= \frac{p^2-1}{24}$. In \cite[Proposition 3]{Graville1996}, Granville and Ono proved similar congruences, namely
	\begin{align*}
		a_{5^j}(5^jn-\delta_{5,j})&\equiv 0\pmod {5^j},\\
		a_{7^j}(7^jn-\delta_{7,j})&\equiv 0\pmod {7^{\lfloor \frac j2 \rfloor +1}},\\
		a_{11^j}(11^jn-\delta_{11,j})&\equiv 0\pmod {11^j}
	\end{align*}
	where $n$, $j$ are positive integers and $\delta_{p,j}\equiv \frac{1}{24}\pmod{p^j}$ for $p \in \{5,7,11\}$.

	In 2019, Hirschhorn and Sellers \cite{Hirs2019} proved a parity result for $a_3(n)$, i.e. for all $n \geq 0$,
	\begin{align*}
		a_3 (n)=\begin{cases}
			1\pmod 2	& \textrm{if }n= 3r^2+2r \textrm{ for some integer } r,\\
			0 \pmod 2	& \textrm{otherwise.}
		\end{cases}
	\end{align*}
 Motivated by this result, both the authors proved that for a non-negative integer $\alpha,$ $a_{3^{\alpha} m}(n)$ is almost always divisible by arbitrary power of $2$ and $3.$ Moreover, they also proved that $a_{t}(n)$ is almost always divisible by arbitrary power of $p_i^j,$ where $j$  is a fixed positive integer and $t= p_1^{a_1}p_2^{a_2}\ldots p_m^{a_m}$ with primes $p_i \geq 5.$ In this following theorem, we obtain infinite families of congruences modulo $2$ for $a_2(n)$ and $a_{13}(n)$ by using Hecke eigen form theory.
		\begin{theorem}\label{mainthm3}
		Let $k$ and $n$ be non-negative integers. For each $1 \leq i \leq k+1,$ let $p_1, p_2,\ldots,p_{k+1}$ be prime numbers such that $p_i \geq 5$. Then for any integer $j \not \equiv 0 \pmod {p_{k+1}},$ we have
		\begin{itemize}
			\item[(i)] $a_2 \left(  p_1^2 p_2^2 \cdots p_{k+1}^2 n + \frac{p_1^2 p_2^2 \cdots p_k^2p_{k+1}\left(8j+ p_{k+1}\right) -1}{8} \right) \equiv 0 \pmod2.$
			\item[(ii)] $a_{13} \left( 104 p_1^2 p_2^2 \cdots p_{k+1}^2 n + 13p_1^2 p_2^2 \cdots p_k^2p_{k+1}\left(\epsilon_pj+ p_{k+1}\right) -7 \right) \equiv 0 \pmod2$
			where
			\begin{align*}
				\epsilon_p=\begin{cases}
					1 & \textrm { if } p \not \equiv 1 \pmod 8,\\
					8 & \textrm { if } p \equiv 1 \pmod 8.
				\end{cases}
			\end{align*}
		\end{itemize}
	\end{theorem}
	
 	\begin{coro}\label{coro3}
		Let $n$ and $k$ be non-negative integers. For a prime $p \geq 5$ and an integer $j \not \equiv 0 \pmod {p}$, we have
		\begin{itemize}
			\item[(i)]$	a_{2}\left( p^{2(k+1)}n + p^{2k+1}j+ \frac{p^{2k+2}-1}{8} \right) \equiv 0 \pmod 2.$
			\item[(ii)] $a_{13}\left(104p^{2k+2}n +13\epsilon_pp^{2k+1} j+13p^{2k+2}-7\right)\equiv 0 \pmod 2.$
		\end{itemize}	
	\end{coro}
	
	Furthermore, we prove the following multiplicative formulae for $2$-core partitions and $13$-core partitions modulo $2$.
	\begin{theorem}\label{mainthm4}
		Let $k$ be a positive integer and $p$ be a prime number such that $ p \equiv 7 \pmod 8.$ Let $r$ be a non-negative integer such that $p$ divides $8r+7,$ then
		\begin{itemize}
			\item[(i)] $a_{2}\left(p^{k+1}n+pr+ \frac{7p-1}{8}\right)\equiv  (-1)  \left(\frac{-2}{p}\right) a_{2}\left(p^{k-1}n+ \frac{8r+7-p}{8p}\right) \pmod2.$
			\item[(ii)] $a_{13}\left(104 p^{k+1}n+104pr+ 91p-7\right)\equiv  (-1) \left(\frac{-2}{p}\right) a_{13}\left(104p^{k-1}n+ \frac{104r+91}{p}-7\right) \pmod2.$
		\end{itemize}
	\end{theorem}
	\begin{coro}\label{coro4}
		Let $k$ be a positive integer and $p$ be a prime number such that $p \equiv 7 \pmod 8.$ Then
		\begin{itemize}
			\item[(i)] $a_{2}\left( p^{2k}n+ \frac{p^{2k}-1}{8}\right) \equiv \left(-1\right)^k \left(\frac{-2}{p}\right)^k a_{2}(n) \pmod2.$
			\item [(ii)] $a_{13}\left(104 p^{2k}n+ 13p^{2k}-7\right) \equiv \left(-1\right)^k \left(\frac{-2}{p}\right)^k a_{13}(104n+6) \pmod2.$
		\end{itemize}
		
	\end{coro}

	
	\section{Preliminaries}
	We recall some basic facts and definition on modular forms. For more details, we refer to \cite{Koblitz, Ono2004}. We start with some matrix groups. We define
	\begin{align*}
		\Gamma:=\mathrm{SL_2}(\mathbb{Z})= &\left\{ \begin{bmatrix}
			a && b \\c && d
		\end{bmatrix}: a, b, c, d \in \mathbb{Z}, ad-bc=1 \right\},\\
		\Gamma_{\infty}:= &\left\{\begin{bmatrix}
			1 &n\\ 0&1	\end{bmatrix}: n \in \mathbb{Z}\right\}.
	\end{align*}
	For a positive integer $N$, we define
	\begin{align*}
		\Gamma_{0}(N):=& \left\{ \begin{bmatrix}
			a && b \\c && d
		\end{bmatrix} \in \mathrm{SL_2}(\mathbb{Z}) : c\equiv0 \pmod N \right\},\\
		\Gamma_{1}(N):=& \left\{ \begin{bmatrix}
			a && b \\c && d
		\end{bmatrix} \in \Gamma_{0}(N) : a\equiv d  \equiv 1 \pmod N \right\}
	\end{align*}
	and 
	\begin{align*}
		\Gamma(N):= \left\{ \begin{bmatrix}
			a && b \\c && d
		\end{bmatrix} \in \mathrm{SL_2}(\mathbb{Z}) : a\equiv d  \equiv 1 \pmod N,\  b \equiv c  \equiv 0 \pmod N \right\}.
	\end{align*}
	A subgroup $\Gamma$ of $\mathrm{SL_2}(\mathbb{Z})$ is called a congruence subgroup if it contains $ \Gamma(N) $ for some positive integer $N$ and the smallest $N$ with this property is called its level. Note that $ \Gamma_{0}(N)$ and $ \Gamma_{1}(N)$ are congruence subgroups of level $N,$ whereas $ \mathrm{SL_2}(\mathbb{Z}) $ and $\Gamma_{\infty}$ are congruence subgroups of level $1.$ The index of $\Gamma_0(N)$ in $\Gamma$ is 
	\begin{align*}
		[\Gamma:\Gamma_0(N)]=N\prod\limits_{p|N}\left(1+\frac 1p\right)
	\end{align*}
	where $p$ runs over prime divisors of $N$.

	Let $\mathbb{H}$ denote the upper half of the complex plane. The group 
	\begin{align*}
		\mathrm{GL_2^{+}}(\mathbb{R}):= \left\{ \begin{bmatrix}
			a && b \\c && d
		\end{bmatrix}: a, b, c, d \in \mathbb{R}, ad-bc>0 \right\}
	\end{align*}
	acts on $\mathbb{H}$ by $ \begin{bmatrix}
		a && b \\c && d
	\end{bmatrix} z = \frac{az+b}{cz+d}.$ We identify $\infty$ with $\frac{1}{0}$ and define $ \begin{bmatrix}
		a && b \\c && d
	\end{bmatrix} \frac{r}{s} = \frac{ar+bs}{cr+ds},$ where $\frac{r}{s} \in \mathbb{Q} \cup \{ \infty\}$. This gives an action of $\mathrm{GL_2^{+}}(\mathbb{R})$ on the extended half plane $\mathbb{H}^{*}=\mathbb{H} \cup \mathbb{Q} \cup \{\infty\}$. Suppose that $\Gamma$ is a congruence subgroup of $\mathrm{SL_2}(\mathbb{Z})$. A cusp of $\Gamma$ is an equivalence class in $\mathbb{P}^{1}=\mathbb{Q} \cup \{\infty\}$ under the action of $\Gamma$.
	
	The group $\mathrm{GL_2^{+}}(\mathbb{R})$ also acts on functions $f:\mathbb{H} \rightarrow \mathbb{C}$. In particular, suppose that $\gamma=\begin{bmatrix}
		a && b \\c && d
	\end{bmatrix}\in \mathrm{GL_2^{+}}(\mathbb{R})$. If $f(z)$ is a meromorphic function on $\mathbb{H}$ and $k$ is an integer, then define the slash operator $|_{k}$ by
	\begin{align*}
		(f|_{k} \gamma)(z):= (\det \gamma)^{k/2} (cz+d)^{-k} f(\gamma z).
	\end{align*}
	
	\begin{definition}
		Let $\Gamma$ be a congruence subgroup of level $N$. A holomorphic function $f:\mathbb{H} \rightarrow \mathbb{C}$ is called a modular form of integer weight $k$ on $\Gamma$ if the following hold:
		\begin{enumerate}[$(1)$]
			\item For all $z \in \mathbb{H}$ and $\begin{bmatrix}
				a && b \\c && d
			\end{bmatrix}\in \Gamma$,
			\begin{align*}
				f \left( \frac{az+b}{cz+d}\right)=(cz+d)^{k} f(z).
			\end{align*}
			
			\item If $\gamma\in SL_2 (\mathbb{Z})$, then $(f|_{k} \gamma)(z)$ has a Fourier expnasion of the form
			\begin{align*}
				(f|_{k} \gamma)(z):= \sum \limits_{n\geq 0}a_{\gamma}(n) q_N^{n}
			\end{align*}
			where $q_N:=e^{2\pi i z /N}$.
		\end{enumerate}
	\end{definition}
	For a positive integer $k$, the complex vector space of modular forms of weight $k$ with respect to a congruence subgroup $\Gamma$ is denoted by $M_{k}(\Gamma)$.
	
	\begin{definition} \cite[Definition 1.15]{Ono2004}
		Let $\chi$ be a Dirichlet character modulo $N$. We say that a modular form $f \in M_{k}(\Gamma_1(N))$ has Nobentypus character $\chi$ if 
		\begin{align*}
			f \left( \frac{az+b}{cz+d}\right)=\chi(d) (cz+d)^{k} f(z)
		\end{align*}
		for all $z \in \mathbb{H}$ and $\begin{bmatrix}
			a && b \\c && d
		\end{bmatrix}\in \Gamma_{0}(N)$. The space of such modular forms is denoted by $M_{k}(\Gamma_0(N), \chi)$.
	\end{definition}
	
	The relevant modular forms for the results obtained in this article arise from eta-quotients. We recall the Dedekind eta-function $\eta (z)$ which is defined by 
	\begin{align}\label{2e1}
		\eta (z):= q^{1/24}(q;q)_{\infty}=q^{1/24} \prod\limits_{n=1}^{\infty} (1-q^n)
	\end{align}
	where $q:=e^{2\pi i z}$ and $z \in \mathbb{H}$. A function $f(z)$ is called an eta-quotient if it is of the form
	\begin{align*}
		f(z):= \prod\limits_{\delta|N} \eta(\delta z)^{r_{\delta}}
	\end{align*}
	where $N$ and $r_{\delta}$ are integers with $N>0$. 
	
	\begin{theorem} \cite[Theorem 1.64]{Ono2004} \label{thm2.1}
		If $f(z)=\prod\limits_{\delta|N} \eta(\delta z)^{r_{\delta}}$ is an eta-quotient such that $k= \frac 12$ $\sum_{\delta|N} r_{\delta}\in \mathbb{Z}$, 
		\begin{align*}
			\sum\limits_{\delta|N} \delta r_{\delta} \equiv 0\pmod {24}	\quad \textrm{and} \quad \sum\limits_{\delta|N} \frac{N}{\delta}r_{\delta} \equiv 0\pmod {24},
		\end{align*}
		then $f(z)$ satisfies
		\begin{align*}
			f \left( \frac{az+b}{cz+d}\right)=\chi(d) (cz+d)^{k} f(z)
		\end{align*}
		for each $\begin{bmatrix}
			a && b \\c && d
		\end{bmatrix}\in \Gamma_{0}(N)$. Here the character $\chi$ is defined by $\chi(d):= \left(\frac{(-1)^{k}s}{d}\right)$ where $s=\prod_{\delta|N} \delta ^{r_{\delta}}$.
	\end{theorem}
	
	\begin{theorem} \cite[Theorem 1.65]{Ono2004} \label{thm2.2}
		Let $c,d$ and $N$ be positive integers with $d|N$ and $\gcd(c,d)=1$. If $f$ is an eta-quotient satisfying the conditions of Theorem \ref{thm2.1} for $N$, then the order of vanishing of $f(z)$ at the cusp $\frac{c}{d}$ is
		\begin{align*}
			\frac{N}{24}\sum\limits_{\delta|N} \frac{\gcd(d, \delta)^2 r_{\delta}}{\gcd(d, \frac{N}{ d} )d \delta}.
		\end{align*}
	\end{theorem}
	Suppose that $f(z)$ is an eta-quotient satisfying the conditions of Theorem \ref{thm2.1} and that the associated weight $k$ is a positive integer. If $f(z)$ is holomorphic at all of the cusps of $\Gamma_0(N)$, then $f(z) \in M_{k}(\Gamma_0(N), \chi)$. Theorem \ref{thm2.2} gives the necessary criterion for determining orders of an eta-quotient at cusps. In the proofs of our results, we use Theorems \ref{thm2.1} and \ref{thm2.2} to prove that $f(z) \in M_{k}(\Gamma_0(N), \chi)$ for certain eta-quotients $f(z)$ we consider in the sequel.

	We recall the definition of Hecke operators and a few relevant results. Let $m$ be a positive integer and $f(z)= \sum \limits_{n= 0}^{\infty}a(n) q^{n}\in M_{k}(\Gamma_0(N), \chi)$. Then the action of Hecke operator $T_m$ on $f(z)$ is defined by
	\begin{align*}
		f(z)|T_{m} := \sum \limits_{n= 0}^{\infty} \left(\sum \limits_{d|\gcd(n,m)} \chi(d) d^{k-1} a\left(\frac{mn}{d^2}\right)\right)q^{n}.
	\end{align*}
	In particular, if $m=p$ is a prime, we have
	\begin{align*}
		f(z)|T_p := \sum \limits_{n= 0}^{\infty}\left( a(pn) + \chi(p) p^{k-1} a\left(\frac{n}{p}\right)\right)q^{n}.
	\end{align*}
	We note that $a(n)=0$ unless $n$ is a non-negative integer.
		\section{Proofs of Theorem \ref{mainthm3} and \ref{mainthm4}}
	
	\subsection{Prelude to the proofs}
	We define
	\begin{small}
		\begin{equation}\label{4e1}
			\sum_{n=1}^{\infty}  b(n) q^n=	q(q^8;q^8)_{\infty}(q^{16};q^{16})_{\infty} \quad \textrm{and} \quad \sum_{n=0}^{\infty}c(n) q^n:= (q;q)_{\infty}^3.
		\end{equation}
	\end{small}
	If $p\nmid n$, then we set $b\left(\frac{n}{p}\right)=0$ and $c\left(\frac{n}{p}\right)=0$. We have the following result.
	\begin{lemma}\label{eta-quotient}
		For $n\geq0$ and for a prime $p\not \equiv 1 \pmod 8$, we have 
		\begin{small}
			\begin{align}\label{4e2}
				b(pn)=(-1)\left(\frac{-2}{p}\right)b\left(\frac{n}{p}\right).
			\end{align}
		\end{small}
		Further if $j \not \equiv 0 \pmod p$, then 
		\begin{small}
			\begin{align}\label{4e3}
				b(p^2n+pj)=0.
			\end{align}
		\end{small}
	\end{lemma}
	\begin{proof} Let $p$ be a prime with $p\not \equiv 1 \pmod 8$. Using \eqref{2e1}, we note that
		\begin{small}
			\begin{equation*}
				\sum_{n=1}^{\infty}  b(n) q^n=	\eta(8z) \eta(16z).
			\end{equation*}
		\end{small}
		By using Theorem \ref{thm2.1}, we obtain that $\eta(8z) \eta(16z) \in S_1(\Gamma_{0}(128), \left(\frac{-128}{\bullet}\right)). $ Thus $\eta(8z) \eta(16z) $ has the Fourier expansion given by
		\begin{small}
			\begin{equation*}
				\sum_{n=1}^{\infty}  b(n) q^n=	\eta(8z) \eta(16z)= q-q^9- 2 q^{17}+ \cdots.
			\end{equation*}
		\end{small}
		Therefore, $b(n)=0$ for all $n \geq 0$ with $n \not \equiv 1 \pmod 8$. Since $ \eta(8z) \eta(16z)$ is a Hecke eigenform, we obtain from \cite[Table 1]{Martin1996} that
		$$\eta(8z) \eta(16z)|T_p= \sum_{n=1}^{\infty} \left(b(pn)+ \left(\frac{-128}{p}\right) b\left(\frac{n}{p}\right)\right) q^n = \lambda(p)\sum_{n=1}^{\infty} b(n)q^n.$$ Note that $ \left(\frac{-128}{p}\right) = \left(\frac{-2}{p}\right) $. Comparing the coefficients of $q^n$ on both sides of the above equation, we get
		\begin{small}
			\begin{equation}\label{4e4}
				b(pn)+ \left(\frac{-2}{p}\right) b\left(\frac{n}{p}\right) = \lambda(p) b(n).
			\end{equation}
		\end{small}
		Since $b(1)=1$ and $b(\frac{1}{p})=0,$ by substituting $n=1$ in the above expression, we get $b(p)=\lambda(p).$ Further, since $b(p)=0$, we obtain that $\lambda(p)=0$. Hence, we conclude from \eqref{4e4} that 
		\begin{small}
			\begin{equation}\label{4e5}
				b(pn)+ \left(\frac{-2}{p}\right) b\left(\frac{n}{p}\right) =0.
			\end{equation}
		\end{small}
		which proves \eqref{4e2}. For $j \not \equiv 0 \pmod p$, replacing $n$ by $pn+j$ in \eqref{4e5}, we get $b(p^2n+ pj)=0$ which proves \eqref{4e3}.
	\end{proof}
	\begin{lemma}\label{qq3}
		For $n\geq0$ and for a prime $p\equiv 1 \pmod 4$, we have 
		\begin{small}
			\begin{align}\label{4e6}
				c\left(p^2n+\frac{p^2-1}{8}\right)=pc(n).
			\end{align}
		\end{small}
		If $p\nmid n$, then 
		\begin{small}
			\begin{align}\label{4e7}
				c\left(pn+\frac{p^2-1}{8}\right)=0.
			\end{align}
		\end{small}
	\end{lemma}
	\begin{proof}
		From \cite[Page 39, Entry 24(ii)]{Berndt1991}, we have
		\begin{small}
			\begin{align*}
				(q;q)_{\infty}^3=\sum_{n=0}^{\infty} (-1)^n (2n+1) q^{\frac{n(n+1)}{2}}.
			\end{align*}
		\end{small}
		Thus 
		\begin{small}
			\begin{align*}
				c(n)=\sum_{\substack{k=0\\\frac{k(k+1)}{2}=n}}^{\infty} (-1)^k (2k+1)=\sum_{\substack{k=0\\(2k+1)^2=8n+1}}^{\infty} (-1)^k (2k+1).
			\end{align*}
		\end{small}
		This implies
		\begin{small}
			\begin{align*}
				c\left(pn+\frac{p^2-1}{8}\right)=\sum_{\substack{k=0\\(2k+1)^2=8pn+p^2}}^{\infty} (-1)^k (2k+1).
			\end{align*}
		\end{small}
		Note that if $(2k+1)^2=8pn+p^2$, then $p|(2k+1)$ and therefore we can write $2k+1=p(2k^{'}+1)$ for some positive integer $k^{'}$. Further for such $k$, we have $k=\frac{2k+1}{2}-\frac{1}{2}=\frac{p(2k^{'}+1)}{2}-\frac{1}{2}=pk^{'}+\frac{p-1}{2}$ which gives $(-1)^k=(-1)^{k^{'}}$. Hence
		\begin{small}
			\begin{align*}
				c\left(pn+\frac{p^2-1}{8}\right)=p\sum_{\substack{k=0\\(2k^{'}+1)^2=8\frac{n}{p}+1}}^{\infty} (-1)^{k^{'}} (2k^{'}+1)=pc\left(\frac{n}{p}\right).
			\end{align*}
		\end{small}
		Replacing $n$ by $pn$, we obtain \eqref{4e6}. Also, \eqref{4e7} follows since $c\left(\frac{n}{p}\right)=0$ if $p\nmid n$.  This completes the proof.
	\end{proof}
	
	We recall the following identity for $13$-core partitions obtained by Kuwali Das.
	\begin{lemma}\cite[Theorem 1]{Das2016}\label{lem3}
		We have
		\begin{small} 
			\begin{equation*}
				\sum_{n=0}^{\infty} a_{13}\left(104n+6\right)q^n \equiv \left(q;q\right)_{\infty} \left(q^2;q^2\right)_{\infty} \pmod2.
			\end{equation*}
		\end{small}
	\end{lemma}

	\begin{lemma}
		For $j\not \equiv 0 \pmod p$ and $n\geq 0$, we have
		\begin{small}
			\begin{align}
				&	a_{2}\left(p^2n +p j +\frac{p^2-1}{8}\right)\equiv 0\pmod 2,\label{4e8}\\
				&	a_{2}\left(p^{2}n +\frac{p^2-1}{8}\right)\equiv \delta_p a_{2} (n)\pmod 2  \label{4e9}\\
				&	a_{13}\left(104p^{2}n +13 p (\epsilon_p j +p)-7\right)\equiv0 \pmod 2,\label{4e10}\\
				&	a_{13}\left(104p^{2}n +13p^{2}-7\right)\equiv \delta_p a_{13} (104n+6)\pmod 2  \label{4e11}
			\end{align}
		\end{small}
		where
		\begin{small}
			\begin{align*}
				\epsilon_p=\begin{cases}
					1 & \textrm { if } p \not \equiv 1 \pmod 8,\\
					8 & \textrm { if } p \equiv 1 \pmod 8.
				\end{cases}
				\quad \textrm{and} \quad
				\delta_p=\begin{cases}
					(-1)\left(\frac{-2}{p}\right) & \textrm { if } p \not \equiv 1 \pmod 8,\\
					p & \textrm { if } p \equiv 1 \pmod 8.
				\end{cases}
			\end{align*}
		\end{small}
	\end{lemma}
	\begin{proof}
		We consider the two cases $p \not\equiv 1 \pmod 8$ and $p \equiv 1 \pmod 8$ separately as follows.
		
		\noindent\textbf{Case 1: $p \not \equiv 1 \pmod 8.$}\\
		From \eqref{eq1}, we have
		\begin{small}
			\begin{equation*}
				\sum_{n=0}^{\infty} a_{2}(n)q^{n} \equiv \frac{(q^2;q^2)_{\infty}^2}{(q;q)_{\infty}} \equiv (q;q)_{\infty} (q^{2};q^{2})_{\infty}  \pmod2
			\end{equation*}
		\end{small}
		Thus using Lemma \ref{lem3}, we yield
		\begin{small}
			\begin{equation}\label{4e12}
				\sum_{n=0}^{\infty} a_{2}(n)q^{8n+1} \equiv \sum_{n=0}^{\infty} a_{13}(104n+6)q^{8n+1} \equiv q (q^8;q^8)_{\infty} (q^{16};q^{16})_{\infty}  \pmod2.
			\end{equation}
		\end{small}
		From \eqref{4e1} and \eqref{4e12}, we get
		\begin{small}
			\begin{equation}\label{4e13}
				a_{2}(n) \equiv a_{13}(104n+6) \equiv b(8n+1) \pmod 2.
			\end{equation}
		\end{small}
		Let $r \not \equiv 0 \pmod p$. From \eqref{4e3}, we have
		\begin{small}
			\begin{align*}
				b(p^2n+pr)=0.
			\end{align*}
		\end{small} Replacing $n$ by $8n-pr+1$, we obtain
		\begin{align*}
			b(8p^2n-p^3r+p^2+pr)=0.
		\end{align*} 
		Note that $8p^2n-p^3r+p^2+pr=8(p^2n-pr\frac{p^2-1}{8}+\frac{p^2-1}{8})+1$. Therefore using \eqref{4e13}, we obtain
		\begin{small}
			\begin{equation}\label{4e14}
				a_{2}\left(p^2n - pr\frac{p^2-1}{8} +\frac{p^2-1}{8}\right) \equiv a_{13}\left(104p^2n - 13 pr(p^2-1) + 13p^2-7\right) \equiv 0 \pmod2.
			\end{equation}
		\end{small}
		Since $\gcd(\frac{p^2-1}{8}, p)=1$ and $\gcd(p^2-1, p)=1,$ when $r$ runs over a residue system excluding the multiples of $p$, so do $ \frac{-r(p^2-1)}{8}$ and $-r(p^2-1)$. Thus for $j \not \equiv 0 \pmod p$, \eqref{4e14} can be written as
		\begin{small}
			\begin{equation*}
				a_{2}\left(p^2n +p j +\frac{p^2-1}{8}\right) \equiv 0 \pmod2
			\end{equation*}
		\end{small}
		and
		\begin{small}
			\begin{equation*}
				a_{13}\left(104p^2n+13 pj+ 13p^2-7\right) \equiv 0 \pmod2.
			\end{equation*}
		\end{small}
		This proves \eqref{4e8} and \eqref{4e10} in the case of $p\not\equiv 1 \pmod 8$.
		
		Next, substituting $n$ by $8pn+p$ in \eqref{4e2}, we obtain
		\begin{small}
			\begin{equation}\label{4e15}
				b(8p^2n+ p^2) = (-1) \left(\frac{-2}{p}\right) b\left(8n+1\right).
			\end{equation}
		\end{small}
		Note that $8p^2n+ p^2=8(p^2n+ \frac{p^2-1}{8})+1$. Therefore using \eqref{4e13} in \eqref{4e15}, we get
		\begin{small}
			\begin{equation*}
				a_{2}\left(p^2n+ \frac{p^2-1}{8}\right) \equiv (-1) \left(\frac{-2}{p}\right) a_2 (n) \pmod 2
			\end{equation*}
		\end{small}
		and
		\begin{small}
			\begin{equation*}
				a_{13}\left(104p^2n+ 13p^2-7\right) \equiv (-1)  \left(\frac{-2}{p}\right) a_{13} \left(104n+6\right) \pmod 2.
			\end{equation*}
		\end{small}
		which proves \eqref{4e9} and \eqref{4e11} in the case of $p\not\equiv 1 \pmod 8$.
		
		\noindent\textbf{Case 2: $p \equiv 1 \pmod 8.$}\\
		From \eqref{eq1}, we have
		\begin{small}
			\begin{equation}\label{4e16}
				\sum_{n=0}^{\infty} a_{2}(n)q^{n} \equiv \frac{(q^2;q^2)_{\infty}^2}{(q;q)_{\infty}} \equiv  (q;q)^3_{\infty} \pmod2.
			\end{equation}
		\end{small}
		From Lemma \ref{lem3}, we have
		\begin{small}
			\begin{equation}\label{4e17}
				\sum_{n=0}^{\infty} a_{13}(104n+6)q^{n} \equiv  (q;q)_{\infty}(q^2;q^2)_{\infty} \equiv (q;q)^3_{\infty} \pmod2.
			\end{equation}
		\end{small}
		Invoking \eqref{4e1}, \eqref{4e16} and \eqref{4e17}, we yield
		\begin{small}
			\begin{equation}\label{4e18}
				a_{2}(n)\equiv 	a_{13}(104n+6) \equiv c(n) \pmod 2.
			\end{equation}
		\end{small}
		If $p\nmid n$, then from \eqref{4e7} and \eqref{4e18}, we get
		\begin{small}
			\begin{align*}
				a_{2}\left(pn+\frac{p^2-1}{8}\right)
				\equiv a_{13}\left(104 pn+13p^2-7 \right) \equiv c\left(pn+\frac{p^2-1}{8}\right)
				\equiv 0\pmod 2.
			\end{align*}
		\end{small}
		Next replacing $n$ by $pn +j$ for $j\not \equiv 0 \pmod p$, we obtain
		\begin{small}
			\begin{align*}
				a_{2}\left(p^2n +p j +\frac{p^2-1}{8}\right)\equiv a_{13}\left(104p^2n +104p j +13p^2-7\right) \equiv 0\pmod 2.
			\end{align*}
		\end{small}
		which proves \eqref{4e8} and \eqref{4e10} in the case of $p\equiv 1 \pmod 8$.
		
		Next using \eqref{4e6} and \eqref{4e18}, we get
		\begin{small}
			\begin{align*}
				a_{2}\left(p^2n+\frac{p^2-1}{8}\right) &\equiv a_{13}\left(104p^2n+13p^2-7\right)\\
				&\equiv c\left(p^2n+\frac{p^2-1}{8}\right)\\
				&\equiv pc(n) \\
				&\equiv pa_{2}(n) \\
				& \equiv pa_{13}(104n+6) \pmod 2.
			\end{align*}
		\end{small}
		which proves \eqref{4e9} and \eqref{4e11} in the case of $p\equiv 1 \pmod 8$.
	\end{proof}
	
	\subsection{Proof of Theorem \ref{mainthm3}(i)} For $1\leq i \leq k-1$, we note that
	\begin{small}
		\begin{align*}
			p_{i}^{2}p_{i+1}^{2}\cdots p_k^{2}n +\frac{p_i^{2}p_{i+1}^{2}\cdots p_k^{2}-1}{8}=p_i^2\left(p_{i+1}^{2}\cdots p_k^{2}n +\frac{p_{i+1}^{2}\cdots p_k^{2}-1}{8}\right)+\frac{p_i^2-1}{8}
		\end{align*}
	\end{small}
	Thus for $1\leq i \leq k-1$, using \eqref{4e9} for $p=p_i$ we have
	\begin{small}
		\begin{align*}
			a_{2}\left(p_{i}^{2}p_{i+1}^{2}\cdots p_k^{2}n +\frac{p_i^{2}p_{i+1}^{2}\cdots p_k^{2}-1}{8}\right)\equiv \delta_{p_i}a_{2}\left(p_{i+1}^{2}\cdots p_k^{2}n +\frac{p_{i+1}^{2}\cdots p_k^{2}-1}{8}\right)\pmod 2.
		\end{align*}
	\end{small}
	Also from \eqref{4e9}, we have
	\begin{small}
		\begin{align*}
			a_{2}\left(p_k^{2}n +\frac{p_k^{2}-1}{8}\right)\equiv \delta_{p_k} a_{2} (n)\pmod 2.
		\end{align*}
	\end{small}
	Therefore from the congruences in the above two displays, we get
	\begin{small}
		\begin{align*}
			a_2\left(p_1^{2}p_2^{2}\cdots p_k^{2}n +\frac{p_1^{2}p_2^{2}\cdots p_k^{2}-1}{8}\right)
			&\equiv \delta_{p_1}\delta_{p_2} \cdots \delta_{p_k}a_2(n) \pmod 2.
		\end{align*}
	\end{small}
	Replacing  $n$ by $p_{k+1}^2n +\frac{p_{k+1} (8j + p_{k+1})-1}{8}$ in the above expression and then using \eqref{4e8} for $p=p_{k+1}$, we get
	\begin{small}
		\begin{align*}
			&a_2\left(p_1^{2}p_2^{2}\cdots p_k^{2}p_{k+1}^2n +\frac{p_1^{2}p_2^{2}\cdots p_k^{2}p_{k+1}(8 j+p_{k+1})-1}{8}\right)\\
			&\hspace{7cm}\equiv \delta_{p_1}\delta_{p_2} \cdots \delta_{p_k}a_2\left(p_{k+1}^2n +p_{k+1} j +\frac{p_{k+1}^2-1}{8}\right) \\
			&\hspace{7cm}\equiv 0\pmod 2.
		\end{align*}
	\end{small}
	when $j \not \equiv 0 \pmod{ p_{k+1}}$. This completes the proof of Theorem \ref{mainthm3}(i).

	\subsection{Proof of Theorem \ref{mainthm3}(ii)}
	
	The proof is similar to the proof of Theorem \ref{mainthm3}(i). For $1\leq i \leq k-1$, we note that
	\begin{small}
		\begin{align*}
			104p_{i}^{2}p_{i+1}^{2}\cdots p_k^{2}n +13p_i^{2}p_{i+1}^{2}\cdots p_k^{2}-7=104p_i^2\left(p_{i+1}^{2}\cdots p_k^{2}n +\frac{p_{i+1}^{2}\cdots p_k^{2}-1}{8}\right)+13p_i^2-7
		\end{align*}
	\end{small}
	Thus for $1\leq i \leq k-1$, \eqref{4e11} implies 
	\begin{small} 
		\begin{align*}
			&a_{13}(104p_i^{2}p_{i+1}^{2}\cdots p_k^{2}n +13p_i^{2}p_{i+1}^{2}\cdots p_k^{2}-7)\\
			&\hspace{5cm}\equiv \delta_{p_i} a_{13}\left( 104\left(p_{i+1}^{2}\cdots p_k^{2}n +\frac{p_{i+1}^{2}\cdots p_k^{2}-1}{8}\right)+6 \right)\\
			&\hspace{5cm}\equiv \delta_{p_i}a_{13}(104p_{i+1}^{2}\cdots p_k^{2}n+13p_{i+1}^{2}\cdots p_k^{2}-7)\pmod 2.
		\end{align*}
	\end{small}
	Also from \eqref{4e11}, we have
	\begin{small}
		\begin{align*}
			a_{13}\left(104p_k^{2}n +13p_k^{2}-7\right)\equiv \delta_{p_k} a_{13} (104n+6)\pmod 2
		\end{align*}
	\end{small}
	Therefore from the above two congruences, we get
	\begin{small}
		\begin{align*}
			a_{13}(104p_1^{2}p_2^{2}\cdots p_k^{2}n +13p_1^{2}p_2^{2}\cdots p_k^{2}-7)
			&\equiv \delta_{p_1}\delta_{p_2} \cdots \delta_{p_k}a_{13}(104n+6) \pmod 2
		\end{align*}
	\end{small}
	Replacing  $n$ by $p_{k+1}^2n +\frac{p_{k+1} (\epsilon_{p_{k+1}}j + p_{k+1})-1}{8}$ in the above expression and then using \eqref{4e10}, we get
	\begin{small}
		\begin{align*}
			&a_{13}(104p_1^{2}p_2^{2}\cdots p_k^{2}p_{k+1}^2n +13p_1^{2}p_2^{2}\cdots p_k^{2}p_{k+1}(\epsilon_{p_{k+1}} j+p_{k+1})-7)\\
			&\hspace{4cm}\equiv \delta_{p_1}\delta_{p_2} \cdots \delta_{p_k}a_{13}(104p_{k+1}^2n +13(p_{k+1} (\epsilon_{p_{k+1}}j + p_{k+1})-1)+6) \\
			&\hspace{4cm}\equiv \delta_{p_1}\delta_{p_2} \cdots \delta_{p_k}a_{13}(104p_{k+1}^2n +13(p_{k+1} (\epsilon_{p_{k+1}}j + p_{k+1}))-7) \\
			&\hspace{4cm}\equiv 0\pmod 2.
		\end{align*}
	\end{small}
	when $j \not \equiv 0 \pmod{ p_{k+1}}$. This completes the proof of Theorem \ref{mainthm3}(ii).

	\subsection{Proof of Theorem \ref{mainthm4}}
	For any prime $p \equiv 7 \pmod 8$, we get from \eqref{4e2} that 
	\begin{small}
		\begin{equation*}
			b(pn)= (-1) \left(\frac{-2}{p}\right) b\left(\frac{n}{p}\right).
		\end{equation*}
	\end{small}
	Let $r \not \equiv 0 \pmod p$. Replacing $n$ by $8(p^kn+r)+7$, we obtain
	\begin{small}
		\begin{equation*}
			b(8(p^{k+1}n+pr)+7p)= (-1)  \left(\frac{-2}{p}\right) b\left(\frac{8(p^kn+r)+7}{p}\right).
		\end{equation*}
	\end{small}
	which can be rewritten as
	\begin{small}
		\begin{equation}\label{4e19}
			b\left(8 \left(p^{k+1}n+pr+ \frac{7p-1}{8}\right)+1\right)= (-1) \left(\frac{-2}{p}\right) b\left(8\left(p^{k-1}n+ \frac{8r+7-p}{8p}\right)+1\right).
		\end{equation}
	\end{small}
	We note here that $\frac{7p-1}{8} $ and $\frac{8r+7-p}{8p}$ are integers. Therefore using \eqref{4e13} and \eqref{4e19}, we get
	\begin{small}
		\begin{equation}\label{4e20}
			a_{2}\left(p^{k+1}n+pr+ \frac{7p-1}{8}\right)\equiv  (-1) \left(\frac{-2}{p}\right) a_{2}\left(p^{k-1}n+ \frac{8r+7-p}{8p}\right) \pmod2.
		\end{equation}
	\end{small}
	and
	\begin{small}
		\begin{equation}\label{4e21}
			a_{13}\left(104 p^{k+1}n+104pr+ 91p-7\right)\equiv  (-1) \left(\frac{-2}{p}\right) a_{13}\left(104p^{k-1}n+ \frac{104r+91}{p}-7\right) \pmod2.
		\end{equation}
	\end{small}
	
	\subsection{Proof of Corollary \ref{coro4}}  	Let $p$ be a prime such that $p \equiv 7 \pmod 8.$ Choose a non negative integer $r$ such that $8r+7=p^{2k-1}.$ Substituting $ k$ by $2k-1$ in \eqref{4e20}, we obtain
	\begin{small}
		\begin{align*}
			a_{2}\left( p^{2k}n+ \frac{p^{2k}-1}{8}\right)&\equiv  (-1) \left(\frac{-2}{p}\right)  a_{2}\left(p^{2k-2}n+ \frac{p^{2k-2}-1}{8} \right)\\
			& \equiv \cdots \equiv \left(-1\right)^k \left(\frac{-2}{p}\right)^k a_{2}(n) \pmod2.
		\end{align*}
	\end{small}
	Substituting $ k$ by $2k-1$ in \eqref{4e21}, we obtain
	\begin{small}
		\begin{align*}
			a_{13}\left(104 p^{2k}n+ 13p^{2k}-7\right)&\equiv  (-1) \left(\frac{-2}{p}\right)  a_{13}\left(104p^{2k-2}n+ 13p^{2k-2}-7 \right)\\
			& \equiv \cdots \equiv \left(-1\right)^k \left(\frac{-2}{p}\right)^k a_{13}(104n+6) \pmod2.
		\end{align*}
	\end{small}

	

\end{document}